\documentclass{amsart}

\usepackage{color}
\usepackage{amsmath, amsthm, amssymb}
\usepackage{amsfonts}
\usepackage[ansinew]{inputenc}
\usepackage[dvips]{epsfig}
\usepackage{graphicx}
\usepackage[english]{babel}
 \theoremstyle{plain}
\newtheorem{thm}{Theorem}

\newtheorem{lem}{Lemma}
\newtheorem{prop}{Proposition}
\newtheorem{rem}{Remark}

\begin{document}

\title[Periodic solutions in Abel equations of the second kind]{Periodic solutions with nonconstant sign in Abel equations of the second kind}

\author{Josep M. Olm}
\address{Department of Applied Mathematics IV, Universitat Polit\`ecnica de Catalunya, Av. Victor Balaguer, s/n, 08800 Vilanova i la Geltr\'u, Spain}
\email{josep.olm@upc.edu}

\author{Xavier Ros-Oton}
\address{Departament de Matem\`{a}tica  Aplicada I, Diagonal 647, 08028 Barcelona, Spain}
\email{xavier.ros.oton@upc.edu}

\author{Tere M. Seara}
\address{Departament de Matem\`{a}tica  Aplicada I, Diagonal 647, 08028 Barcelona, Spain}
\email{josep.olm@upc.edu}


\begin{abstract}
The study of periodic solutions with constant sign in the Abel equation of the second kind can be made through the equation of the first kind. This is because the situation is equivalent under the transformation $x \mapsto x^{-1}$, and there are many results available in the literature for the first kind equation. However, the equivalence breaks down when one seeks for solutions with nonconstant sign. This note is devoted to periodic solutions with nonconstant sign in Abel equations of the second kind. Specifically, we obtain sufficient conditions to ensure the existence of a periodic solution that shares the zeros of the leading coefficient of the Abel equation. Uniqueness and stability features of such solutions are also studied.
\end{abstract}

\maketitle

\section{Introduction} \label{intro}

Abel Ordinary Differential Equations (ODE)  of the second kind \cite{p},
\begin{equation} \label{abel2}
\left[b_0(t)+b_1(t)x\right]\dot x = a_0(t)+a_1(t)x+a_2(t)x^2, \quad a_i(t),b_i(t) \in \mathcal
{C}([0,T]),
\end{equation} can be regarded as a generalization of Riccati's equation \cite{p}. This family of equations deserves special interest in the applied mathematics field because it appears in different contexts, running from control problems \cite{fos} to mathematical physics and nonlinear mechanics issues \cite{p,ps}. It is also remarkable that a class of Abel equations of the first kind \cite{p} can be written as (\ref{abel2}) with the change of variables $x \mapsto x^{-1}$.

Indeed, polynomial differential equations of the type
\begin{equation} \label{poli} \dot x = \sum_{i=0}^na_i(t)x^i, \quad a_i(t) \in \mathcal{C}([0,T]),  \quad i=0,\ldots,n,\end{equation}
are also known as Abel-like \cite{t} or generalized Abel \cite{a} equations   because, when $n=3$, (\ref{poli}) is an Abel ODE of the first kind. The existence of periodic solutions in (\ref{poli}), i.e. solutions verifying $x(0)=x(T)$, has attracted considerable research effort: see, for example, \cite{t,a,gl,as,agg,bt,ab,b} and references therein. This is mainly due to its relation with the number of limit cycles of planar polynomial systems and, therefore, with Hilbert's 16th problem \cite{s}. Contrarily, few results regarding periodic solutions in Abel ODE of the second kind have been published \cite{fos,fo,ors}.

Notice that the change of variables $x\mapsto x+b_0/b_1$ allows to recast (\ref{abel2}) as
\begin{equation} \label{normal}
 x\dot x = A(t)+B(t)x+C(t)x^2, \quad A,B,C \in  \mathcal{C}([0,T]).
\end{equation}
The transformation being time-preserving, the study of periodic solutions in (\ref{abel2}) and (\ref{normal})
whenever $b_1$ has constant sign, i.e. $b_1(t) \neq 0$, for all $ t$, is equivalent.

The existence of nontrivial periodic solutions of constant sign in (\ref{normal}) may be carried out after transforming
it into the Abel equation of the first kind
\begin{equation} \label{normalabel}
\dot x = A(t)x^3+B(t)x^2+C(t)x
\end{equation}
using the aforementioned change $x \mapsto x^{-1}$, which keeps
the equivalence between these class of solutions in (\ref{normal})
and (\ref{normalabel}). Different conditions are available in the
literature yielding to upper and/or lower bounds on the
number of periodic solutions of (\ref{normalabel}) or
(\ref{poli})-like Abel ODE with more or less generic coefficients.
The results cover from the simplest case, $A(t) \neq 0$, for all $
t$, in (\ref{normalabel}) \cite{as,fo} or $a_n(t) \neq 0$, for all
$ t$, in (\ref{poli}) \cite{a}, to the most complex in which no
sign condition is assumed on the coefficients \cite{agg,b}, going
through situations in which some of the coefficients, often $A(t)$ or
$a_n(t)$,  are demanded not to change sign \cite{t,gl,as,bt,ab}.


The study of periodic solutions of (\ref{normal}) with nonconstant sign is also a challenging problem that can not be
tackled via the Abel ODE of the first kind and about which, as far as the authors know, no results have been yet reported.
Notice also that if a solution of (\ref{normal}) has nonconstant sign, then its zeros are also zeros of $A(t)$. Hence, the search of periodic solutions in (\ref{normal}) with nonconstant sign only makes sense when $A(t)$ itself has nonconstant sign.

This note deals with the existence of this type of periodic solutions in Abel equations of the second kind. The main result reads as follows:

\begin{thm} \label{teo} Let $A(t),B(t),C(t)$ be $\mathcal C^1$, $T$-periodic functions. If $A(t)$ has at least
one zero in $[0,T]$ and \begin{equation}\label{cond2}\min |B(t)|^2  > -4\min\dot A(t)\cdot \left[1+T\max|C(t)|\right],\end{equation} then (\ref{normal})
has a $T$-periodic solution that has the sign of $-A(t)B(t)$, and it is also $\mathcal C^{1}$.
\end{thm}

Furthermore, such solutions are shown to be the unique $T$-periodic solutions of (\ref{normal}) with nonconstant sign. In some cases, it is proved that there exists only one solution of this type. Also, a stability analysis reveals that these solutions are unstable. These results are later applied to the normal form of the
Abel ODE of the second kind, which is obtained setting $B(t)=1$
and $C(t)=0$ in (\ref{normal}). For this case, restriction
(\ref{cond2}) is shown to be sharp.

The remainder of the paper is organized as follows. Section
\ref{prova} is devoted to the proof of Theorem \ref{teo}. Section
\ref{unies} deals with the  uniqueness and stability of the
periodic solutions of (\ref{normal}) with nonconstant sign. Finally, Section
\ref{sec_normalform} considers the application of the previous
results to the normal form of the Abel ODE of the second kind.

\section{Proof of Theorem \ref{teo}} \label{prova}

Firstly, let us establish a generic and rather straightforward result that will be used in subsequent demonstrations.

\begin{lem} \label{lema1} Consider the ODE
\begin{equation} \label{unaode} \dot x=S(t,x), \quad S: \Omega  \rightarrow \mathbb{R},\end{equation}
where $\Omega:=\mathbb{R} \times \mathbb{R}^\ast$, $\mathbb{R}^\ast= \mathbb{R}\setminus \{0\}$ and $S$ is a locally Lipschitz function. Assume that $m,n \in \mathbb{R}$ and let
$r:=\{(t,x): \ x=mt+n\}$ be a straight line of slope $m$, which splits $\mathbb{R}^2$ into the half planes $\Omega_r^+=\{(t,x):x>mt+n\}$,
$\Omega_r^-=\{(t,x): \ x<mt+n\}$. Finally, let $t_1,t_2 \in \mathbb{R}$, with $t_1 < t_2$.

\vspace{3mm}

\noindent (i) Assume that $S(t,x)>m$ for all $(t,x)\in [t_1,t_2) \times \mathbb{R}^\ast \cap r$. Then, any maximal solution $x(t)$ of
(\ref{unaode}) defined for all $t \in I_\omega \subseteq \mathbb{R}$  with $(t_1,x(t_1)) \in\overline{\Omega_r^+}$ is such that
$(t,x(t))\in\Omega_r^+$, for all $t\in (t_1,t_2) \cap I_\omega$.

\vspace{1mm}

\noindent (ii) Assume that $S(t,x)<m$ for all $(t,x)\in [t_1,t_2) \times \mathbb{R}^\ast \cap r$. Then, any maximal solution $x(t)$ of
(\ref{unaode}) defined for all $t \in I_\omega \subseteq \mathbb{R}$  with $(t_1,x(t_1)) \in\overline{\Omega_r^-}$ is such that
$(t,x(t))\in\Omega_r^-$, for all $t\in (t_1,t_2) \cap I_\omega$.
\end{lem}

\begin{proof} (i) Notice that if $(t_1,x(t_1))\in \overline{\Omega_r^+}$ then, by continuity, for
$t$ close enough to $t_1$ it happens that $(t,x(t))\in{\Omega_r^+}$: for $(t_1,x(t_1))\in \Omega_r^+$ it is rather immediate,
while for $(t_1,x(t_1))\in r$ the claim follows because $\dot x(t_1) > m$.

Assume that $x(t)$ contacts $r$ for the first time in $(t_1,t_2)$ at $t=c > t_1$, i.e. that $(c,x(c)) \in r$. Then, $x(c)=mc+n$, and $x(t)>mt+n$
for $t_1<t<c$, so
\[\dot x(c)=\lim_{t\rightarrow c^-}\frac{x(t)-x(c)}{t-c}\leq
\lim_{t\rightarrow c^-}\frac{mt+n-(mc+n)}{t-c}=m,\] which contradicts the hypothesis $\dot x(c)=S(c,x(c))>m$. Hence, $x(t)$ can not contact
again $r$ and, therefore, it remains in $\Omega_r^+$ for all $t\in (t_1,t_2) \cap I_\omega$.

The proof of (ii) is analogous.
\end{proof}

The hypotheses of Theorem \ref{teo} are assumed to be fulfilled
throughout the rest of the section. Furthermore, the $T$-periodicity and $\mathcal C^1$ character of $A(t)$ implies that $\min \dot A(t)\leq0$. Hence, it is immediate from (\ref{cond2}) that $B(t) \neq 0$ for all $t$. Thus, using
the change of variables $x\mapsto -x$ if necessary, it is no loss
of generality to assume $B(t)>0$, for all $ t$, for the remainder
of the section.

Theorem $\ref{teo}$ considers a case in which $A(t)$ has
nonconstant sign. The next Lemmas study the behavior of the
solutions of (\ref{normal}) in an open interval $(a,b)$ where
$A(t)>0$, $a,b$ being two consecutive zeros of $A(t)$.

\begin{rem} \label{signeg} Notice that there is no loss of generality in
assuming $A(t)>0$ in $(a,b)$ because, otherwise, the change of variables $(t,x) \mapsto (-t,-x)$
reduces (\ref{normal}) to
\[ x\dot x = \hat A(t)+\hat B(t)x+\hat C(t)x^2,\]
with $\hat A(t)=-A(-t)$, $\hat B(t)=B(-t)$ and $\hat C(t)=-C(-t)$, for all $t \in (a,b)$.
\end{rem}

\begin{lem} If (\ref{cond2}) is satisfied, then \begin{equation}\label{cond}\min |B(t)|^2  >  2\max|A(t)|\cdot \max|C(t)|,\end{equation}
\end{lem}

\begin{proof} Recalling that $\min \dot A(t)\leq0$, (\ref{cond2}) yields \[\min |B(t)|^2  > -4\min\dot A(t)\cdot \left[1+T\max|C(t)|\right] \geq -2T\min\dot A(t)\max|C(t)|.\]
Then, it is sufficient to prove that \begin{equation}\label{poinc} -T\min\dot A(t)\geq \max|A(t)|.\end{equation}
For, let $t_0\in \mathbb R$ be such that $|A(t_0)|=\max|A(t)|$, and let $t_1,t_2$ be zeros of $A(t)$ such that $t_0-T<t_1\leq t_0\leq t_2<t_0+T$. Then, applying the Mean Value Theorem, \begin{eqnarray*} A(t_0)&=&A(t_0)-A(t_1)=(t_0-t_1)\dot A(\xi_1)\geq T\min \dot A(t)\\ -A(t_0)&=&A(t_2)-A(t_0)=(t_2-t_0)\dot A(\xi_2)\geq T\min \dot A(t),\end{eqnarray*} from which we deduce (\ref{poinc}) and therefore (\ref{cond}).
\end{proof}

\begin{lem} \label{lemados} Let $a,b \in \mathbb{R}$ be such that $A(a)=A(b)=0$, with $A(t)>0$,
for all $t \in (a,b)$. Then, any negative solution $x(t)$ of
(\ref{normal}) defined on $[t_1,t_2)$, $t_1\geq a$, can be
extended to $[t_1,b)$.
\end{lem}

\begin{proof} The ODE (\ref{normal}) can be written as \begin{equation} \label{recalling} \dot x=S(t,x)=\frac{A(t)}{x}+B(t)+C(t)x\end{equation} in the domain
$\Omega^-:=\mathbb{R} \times \mathbb{R}^-$. Let $x(t)$ be a
solution with $x(t_1)<0$ and maximal interval of definition
$I_\omega=(\omega_-,\omega_+)$, with $\omega_-< t_1$. Let us
assume that $\omega_+< b$ and proceed by contradiction.

As $\omega_+< b$, $t\rightarrow \omega_+$ implies that either
$x(t)\rightarrow -\infty$ or $x(t)\rightarrow 0$. Let us first see
that it is not possible to have $x(t)\rightarrow -\infty$. For,
let us select $M \in \mathbb{R^+}$ and let us define the straight
line $r_M:=\{(t,x): \ x+M=0\}$. If $C \not \equiv 0$, the
selection
\[M=\sqrt{\frac{\max|A(t)|}{\max|C(t)|}}\]
and relation (\ref{cond}) in Lemma 2 indicate that
\begin{equation} \label{Sposit} S(t,x)=-\frac{A(t)}{M}+B(t)+C(t)M>0, \quad \forall \ (t,x)\in
I_\omega \times \mathbb{R}^- \cap r_M.\end{equation} Alternatively, if $C
\equiv 0$, the selection of a sufficiently large $M$ also yields
(\ref{Sposit}). In any case, by Lemma \ref{lema1}.i, $x(t)+M>0$, for
all $t\in(t_1,\omega_+)$, so $x(t) \not \rightarrow -\infty$ for $t\rightarrow \omega_+ < b$.

Let us now see that it cannot be $x(t)\rightarrow 0$ for $t\rightarrow \omega_+ < b$.
For, let us take $c \in \mathbb{R}$, $t_1<c<\omega_+$ and
select $N \in \mathbb{R}^+$ small enough, in such a way that
$$S(t,x)=-\frac{A(t)}{N}+B(t)+C(t)N<0,  \quad \forall t \in \left[c,\omega_+\right].$$
Then, defining $r_N:=\left\{\right(t,x): \ x+N=0\}$, it is
immediate that $S(t,x)<0$, for all $(t,x) \in
\left[c,\omega_+\right] \times \mathbb{R}^- \cap r_N$ and, by
Lemma \ref{lema1}.ii, $x(t)+N<0$ for $t\in(c,\omega_+)$, i.e.
$x(t)\not \rightarrow 0$.

Consequently, it has to be $\omega_+\geq b$ and the solution
$x(t)$ is defined in $[t_1,b)$.
\end{proof}

\begin{lem} \label{lematres} Let $a,b \in \mathbb{R}$ be such that $A(a)=A(b)=0$, with $A(t)>0$,
for all $t \in (a,b)$. Then, there exists a $\mathcal{C}^1$
solution $x^\ast(t)$ of (\ref{normal}) in $(a,b)$, which is
negative and such that
$$x^\ast(t)\rightarrow 0 \quad \mbox{and} \quad \dot x^\ast(t) \rightarrow \frac{B(a)}{2} - \sqrt{\frac{B(a)^2}{4}+\dot A(a)} \quad \mbox{when} \quad t\rightarrow a^+.$$
\end{lem}

\begin{proof} The introduction of a new variable $s$ allows to transform (\ref{normal}) in the following planar, generalized Li\'enard system \cite{lienard}:
\begin{eqnarray} \label{lat} \frac{dt}{ds} & = & x, \\ \label{lax} \frac{dx}{ds} & = & A(t)+B(t)x+C(t)x^2 .\end{eqnarray}
Notice that, when $x \neq 0$, the portrait of the integral
curves of (\ref{normal}) and the phase plane of
(\ref{lat})-(\ref{lax}) are coincident, preserving the orientation if $x>0$ and reversing it if $x<0$.

Since $A(a)=0$ and $A>0$ in $(a,b)$, it results that $\dot A(a)\geq 0$. Let us study these two cases:

\vspace{3mm}

(i) If $\dot A(a)>0$, then $(t,x)=(a,0)$ is a hyperbolic critical
point, indeed a saddle, of (\ref{lat})-(\ref{lax}). The eigenvalues
are:
\[\lambda_{\pm}^a =  \frac{B(a)}{2}\pm \sqrt{\frac{B(a)^2}{4}+\dot A(a)}, \]
with $\lambda_{\pm}^a \in \mathbb{R}$ because of (\ref{cond2}), the
associated invariant subspaces of the linearized system being
\[\mathbb{E}_+^a=\mathbb{E}_u^a=\mbox{span}\left\{
\left(1,\lambda_{+}^a\right)\right\}, \quad
\mathbb{E}_-^a=\mathbb{E}_s^a=\mbox{span}\left\{
\left(1,\lambda_{-}^a\right)\right\}.\] Hence, by the Stable
Manifold Theorem \cite{gh}, there exists a unique $\mathcal{C}^1$
invariant stable manifold, tangent to $\mathbb{E}_s^a$ at $(a,0)$,
with slope $\lambda_-^a$, i.e. lying on the subsets $\mathcal{A}^+:=\{(t,x): \ t<a, \ x > 0\}$ and $\mathcal{A}^-:=\{(t,x):
\ t>a, \ x < 0\}$ when $t
\neq a$. The branch of the manifold that lies in $\mathcal{A}^+$ is a positive, $\mathcal{C}^1$ solution
$x^\ast(t)$ of (\ref{normal}) in $(a-\epsilon,a)$, $\epsilon > 0$, that satisfies
$x^\ast(t)\rightarrow 0$ and $\dot x^\ast(t) \rightarrow
\lambda_-^a$ when $t\rightarrow a^-$. Equivalently, the branch in $\mathcal{A}^-$ is a negative, $\mathcal{C}^1$ solution $x^\ast(t)$ of (\ref{normal}) in $(a,a+\epsilon)$, $\epsilon > 0$, that satisfies
$x^\ast(t)\rightarrow 0$ and $\dot x^\ast(t) \rightarrow
\lambda_-^a$ when $t\rightarrow a^+$.

\vspace{2mm}

(ii) If $\dot A(a)=0$, then $(t,x)=(a,0)$ is a non-hyperbolic
critical point with eigenvalues
\[ \lambda_u^a=B(a), \quad \lambda_c^a=0,\]
the associated invariant subspaces of the linearized system being
\[\mathbb{E}_u^a=\mbox{span}\{ \left(1,B(a)\right)\}, \quad
\mathbb{E}_c^a=\mbox{span}\{ \left(1,0\right)\}.\] Hence, by the
Center Manifold Theorem \cite{k}, there exists a (not necessarily
unique) $\mathcal{C}^1$, invariant center manifold, tangent to
$\mathbb{E}_c$ at $(a,0)$.

Let us finally see that this orbit lies on $\mathcal{A}^-$, which means that
it matches a negative, $\mathcal{C}^1$ solution $x^\ast(t)$ of
(\ref{normal}) that satisfies $x^\ast(t)\rightarrow 0$ and $\dot
x^\ast(t) \rightarrow 0$ when $t\rightarrow a^+$. Let us denote this
orbit as $x=h(t)$, with $h(a)=\dot h(a)=0$ and satisfying
$$h(t)(C(t)h(t)+B(t)-\dot h(t))=-A(t).$$
As $C(a)h(a)+B(a)-\dot h(a)=B(a)>0$, then $C(t)h(t)+B(t)-\dot
h(t)>0$ for $t-a$ small enough; consequently, $h$ and $-A$ have the
same sign in a neighborhood of $(a,0)$, so $h(t)<0$ for $0<t-a<<1$.

\vspace{3mm}

Finally notice that, by Lemma \ref{lemados}, this $\mathcal{C}^1$
solution $x^\ast(t)$ is defined in $(a,b)$, which completes the
proof.
\end{proof}

\begin{lem} \label{lemaquatre} Let $a,b \in \mathbb{R}$ be such that $A(a)=A(b)=0$, with $A(t)>0$,
for all $t \in (a,b)$. Then, there exists a $\mathcal{C}^1$ solution
$x^\ast(t)$ of (\ref{normal}) in $(a,b)$, which has the sign of
$-A(t)$, and is such that
\begin{eqnarray*}
 & & x^\ast(t)\rightarrow 0 \quad \mbox{and} \quad \dot x^\ast(t) \rightarrow \frac{B(a)}{2}- \sqrt{\frac{B(a)^2}{4}+\dot A(a)} \quad \mbox{when} \quad t\rightarrow a^+, \\
 & & x^\ast(t)\rightarrow 0 \quad \mbox{and} \quad \dot x^\ast(t) \rightarrow \frac{B(b)}{2}- \sqrt{\frac{B(b)^2}{4}+\dot A(b)} \quad \mbox{when} \quad t\rightarrow b^-.
\end{eqnarray*}
\end{lem}

\begin{proof} Let $x^\ast(t)$ be the solution of (\ref{normal}) featured in Lemma \ref{lematres}. Then, it remains to be proved the behavior for $t \rightarrow b^-$.

Firstly, for all $t \in (a,b)$, the Mean Value Theorem ensures that
there exists $\xi\in(t,b)$ such that $A(t)=\dot A(\xi)(t-b)$. Let
now $r_b$ be the straight line $r_b:=\{(t,x): \ x=\alpha(t-b)\}$,
where \[\alpha=\sqrt{\frac{-\min\dot A(t)}{1+T\max|C(t)|}}.\] Then,
for all $t\in(a,b)$ such that $(t,x)\in r_b$,
\begin{eqnarray*}S(t,x)&=&\frac{A(t)}{x}+B(t)+C(t)x\\ &\geq& \frac{\dot A(\xi)(t-b)}{\alpha(t-b)}+\min |B(t)|-\alpha T\max|C(t)|\\&\geq&
\alpha+\min |B(t)|-\alpha(1+T\max|C(t)|)+\frac{\min\dot
A(t)}{\alpha}\\&=&\alpha+\min |B(t)|-2\sqrt{-\min\dot
A(t)(1+T\max|C(t)|)}>\alpha,\end{eqnarray*} where in the last inequality we have used condition (\ref{cond2}). As, by Lemma
\ref{lematres}, $x^\ast(t)\rightarrow 0$ when $t\rightarrow a^+$,
Lemma \ref{lema1}.i guarantees that $x^\ast(t)>\alpha(t-b)$ for all
$t\in (a,b)$. But since $x^\ast(t)$ is negative in $(a,b)$, then
$\alpha(t-b)<x^\ast(t)<0$ and, taking limits for $t \rightarrow
b^-$, it is immediate that $x^\ast(t) \rightarrow 0$.

Secondly, consider the equivalent expression of (\ref{normal}) in
terms of the planar, autonomous system (\ref{lat})-(\ref{lax}).
Since $A(b)=0$ and $A(t)>0$ in $(a,b)$, it results that $\dot A(b)\leq
0$. Let us then split the study in two cases:

\vspace{3mm}

(i) If $\dot A(b)<0$, then $(t,x)=(b,0)$ is a hyperbolic critical
point of (\ref{lat})-(\ref{lax}). The eigenvalues are:
\[\lambda_{\pm}^b =  \frac{B(b)}{2}\pm \sqrt{\frac{B(b)^2}{4}+\dot A(b)}, \]
with $\lambda_{\pm}^b \in \mathbb{R}^+$ because of (\ref{cond2}), the
associated invariant subspaces of the linearized system being
\[ \mathbb{R}^2=\mathbb{E}_u^b = \mbox{span}\left\{\left(1,\lambda_-^b\right),\left(1,\lambda_+^b\right)\right\}.\]
 Hence, it is an unstable node and, in $t=b$, all the orbits are tangent to
one of the two eigenvectors that span $\mathbb{E}_u^b$. It is then
immediate that the solution $x^\ast(t)$ of (\ref{normal}), which is
known to satisfy  $x^\ast(t)>\alpha(t-b)$, matches one of these
orbits. Therefore, for $t$ close enough to $b$ it has to be $\dot
x^\ast(t)<\alpha$ and, as it can be easily proved that $\lambda_-^b
< \alpha < \lambda_+^b$, it results that
\[\dot x^\ast(t) \rightarrow \lambda_-^b=\frac{B(b)}{2}- \sqrt{\frac{B(b)^2}{4}+\dot A(b)} \quad \mbox{when} \ \ t \rightarrow b^-.\]

\vspace{2mm}

(ii) If $\dot A(b)=0$, the situation is equivalent to the case $\dot
A(a)=0$ discussed in the proof of Lemma \ref{lematres}. Namely,
$(t,x)=(b,0)$ is a non-hyperbolic critical point with eigenvalues
\[ \lambda_u^b=B(b), \quad \lambda_c^b=0,\]
the associated invariant subspaces of the linearized system being
\[\mathbb{E}_u^b=\mbox{span}\{ \left(1,B(b)\right)\}, \quad
\mathbb{E}_c^b=\mbox{span}\{ \left(1,0\right)\}.\] A
technique similar to the one followed in the preceding item shows
that $\dot x^\ast(t) \rightarrow 0$ when $t \rightarrow b^-$.
\end{proof}

Let us now proceed with the proof of Theorem \ref{teo}. As $A(t)$ has, at least, one zero in $[0,T]$ by hypothesis, let $t_0 \in \mathbb{R}$ be such that $A(t_0)=0$. Then, we define
\[Z:=\{t\in [t_0,t_0+T]: \ A(t)=0\}.\] Let $\mathcal P$ and
$\mathcal N$ denote the sets of maximal intervals in $[t_0,t_0+T]$ where $A(t)$ is
positive and negative, respectively. Let then $I_i=(a_i,b_{i})$,
$a_i,b_{i} \in Z$, denote an interval of $\mathcal P \cup \mathcal
N$. Lemma \ref{lemaquatre} and Remark \ref{signeg} ensure that, for
every $I_i$, there exists a $\mathcal{C}^1$ solution $x_i^\ast(t)$
of (\ref{normal}) on $I_i$, which has the sign of $A(t)$ in $I_i$, and
is such that $x_i^\ast(t) \rightarrow 0$ when $t \rightarrow a_i^+$
and also when $t \rightarrow b_{i}^-$. Hence, the function
constructed as
\begin{equation} \label{lasolu} x^\ast(t)=\left\{ \begin{array}{cll}
x_i^\ast(t) & \textrm{if} & t \in I_i \\ 0 & \textrm{if} & t \in Z,
\end{array} \right.\end{equation} is indeed a continuous solution
of (\ref{normal}) in $\mathbb{R}$ which is also $\mathcal{C}^1$ in
every open interval $I_i$. Let us finally prove that $x^\ast(t)$ is
$\mathcal{C}^1$ for all $t_i \in Z$. Three
different situations need to be considered:

\vspace{3mm}

(i) If $\dot A(t_i)>0$ then the graph of $x^\ast(t)$ in a neighborhood of
$t_i$ is the orbit of the (unique) stable manifold of
(\ref{lat})-(\ref{lax}), so $x^\ast(t)$ is $\mathcal{C}^1$ in $t_i$ (see
the discussion in the proof of Lemma \ref{lematres} for the case
$\dot A(a) > 0$).

\vspace{2mm}

(ii) If $\dot A(t_i)=0$, then the graph of $x^\ast(t)$ in a neighborhood
of $t_i$ is the orbit of a center manifold of
(\ref{lat})-(\ref{lax}), so $x^\ast(t)$ is $\mathcal{C}^1$ in $t_i$ (see
the discussion in the proof  of Lemma \ref{lematres} for the case
$\dot A(a) = 0$).

\vspace{2mm}

(iii) If $\dot A(t_i)<0$, then the graph of $x^\ast(t)$ in a neighborhood
of $t_i$ matches an orbit of (\ref{lat})-(\ref{lax}) with slope
\[ \dot x^\ast(t_i)= \frac{B(t_i)}{2}- \sqrt{\frac{B(t_i)^2}{4}+\dot A(t_i)}, \]
so $x^\ast(t)$ is $\mathcal{C}^1$ in $t_i$ (see the discussion in the proof of Lemma \ref{lemaquatre} for the case $\dot A(b) < 0$).

Finally, the $T$-periodic extension of $x^\ast$ is a $\mathcal C^1$ solution of (\ref{normal}) defined in $\mathbb R$, which completes the proof.
\qed

\begin{rem} Notice that if  $A(t)$ has degenerate zeros, then the construction of the solution $x^\ast(t)$ requires the use of the Center Manifold Theorem. This means that, in such a case, there may exist a family of periodic solutions of (\ref{normal}) with the same sign as $-A(t)B(t)$. \end{rem}

\section{Uniqueness and stability} \label{unies}

The next result reveals that the $T$-periodic solution(s) with non constant sign that arise from Theorem \ref{teo} are unique, in the sense that there do not exist other $T$-periodic solutions with nonconstant sign in (\ref{normal}) sharing only some of the zeros of $A(t)$.

\begin{thm} \label{p1} Let the assumptions of Theorem \ref{teo} be
fulfilled. Then, all $T$-periodic solutions of (\ref{normal}) with nonconstant sign have the sign of $-A(t)B(t)$. Moreover, if all the zeros of $A(t)$ are simple, then
(\ref{normal}) has a unique $\mathcal C^{1}$, $T$-periodic solution
with nonconstant sign.
\end{thm}

\begin{proof} Let $x(t)$ be a solution of (\ref{normal}) with nonconstant sign and non identically $0$ (otherwise the result is trivial). As stated in Section \ref{intro}, the zeros of  $x(t)$ are also zeros of $A(t)$. Then, let $I=(a,b)$ be an interval where $x(t)\neq 0$, for all $t \in I$, and such that $x(a)=x(b)=0$. It is no loss of generality to assume that $x(t)<0$ in $I$ (see Remark \ref{signeg}), which yields $\dot x(a)\leq 0$.
Let us consider three cases:

\vspace{2mm}

(i) If $\dot A(a)>0$, it has been noticed in the proof of Lemma \ref{lematres} that (\ref{lat})-(\ref{lax}) posesses two invariant manifolds in $(a,0)$, a stable one and an unstable one, with negative and positive slopes, respectively. As $x(t)$ has non positive derivative in $t=a$, it has to coincide with the solution curve corresponding to the stable manifold orbit and, consequently, with the periodic solution with nonconstant $x^\ast(t)$ guaranteed by Theorem \ref{teo}, from $t=a$ to the next zero of $A(t)$, and this zero has to be in $t=b$.

\vspace{2mm}

(ii) If $\dot A(a)=0$, it has been noticed in the proof of Lemma \ref{lematres} that (\ref{lat})-(\ref{lax}) possesses two invariant manifolds in $(a,0)$, a (non necessarily unique) center one and an unstable one, with null and positive slopes, respectively. As  $x(t)$ has non positive derivative in $t=a$, it has to coincide with one of the solution curves corresponding to the center manifold and, consequently, with the periodic solution with nonconstant $x^\ast(t)$ guaranteed by Theorem \ref{teo}, from $t=a$ to the next zero of $A(t)$, and this zero has to be in $t=b$.

\vspace{2mm}

(iii) If $\dot A(a) < 0$, it has been noticed in the proof of Lemma \ref{lematres} that (\ref{lat})-(\ref{lax}) possesses an
unstable node in $(a,0)$, and also that all the solutions tending to this point have positive slope. This is in contradiction with $\dot x(a)\leq 0$, which means that there cannot exist a solution of (\ref{normal}) verifying $\dot x(a) \le 0$ and $\dot A (a)<0$.

\vspace{2mm}

It is therefore proved that $x(t)=x^\ast(t)$, for all $t$ such that $x(t) \neq 0$. Furthermore, when $x(t)=0$, then $A(t)=0$ and $x^\ast(t)=0$, which implies that $x(t)=x^\ast(t)$, for all $t \in \mathbb{R}$.

Finally, when all the zeros of $A(t)$ are simple, the uniqueness of the Stable
Manifold (see the proof of Lemma \ref{lematres}) yields the existence of a unique $T$-periodic solution, $x^\ast(t)$, with nonconstant sign.
\end{proof}

The instability of the $T$-periodic solutions of (\ref{normal}) with nonconstant sign is claimed in next Theorem:

\begin{thm} \label{esta} Let the assumptions of Theorem \ref{teo} be
fulfilled. Then, any $\mathcal C^{1}$, $T$-periodic solution of
(\ref{normal}) with nonconstant sign is unstable.
\end{thm}

\begin{proof} Assume that $B(t)>0$, for all $t$, which is no loss of generality. Let $I=(a,b) \in \mathcal{P}$,
i.e. an interval such that $A(t)>0$, for all $t$ (in case that $A(t) \le 0$, for all $t$, use the change of variables $(t,x) \mapsto (-t,-x)$ suggested in Remark \ref{signeg}). Theorems \ref{teo} and \ref{p1} ensure that any $T$-periodic solution of (\ref{normal}) with nonconstant sign has the sign of $-A(t)$. Hence, let $x^\ast(t)$ denote one of these solutions, which satisfies $x^\ast(t)<0$, for all $t \in I$.

Let $x_a \in \mathbb{R}^+$ and consider the straight line
$r_{x_a}:=\left\{(t,x): \ x-x_a =0\right\}$. It is rather
immediate that there exists an open interval $ J \subseteq
\mathbb{R}^+$ such that, for all $x_a \in J$,
\[ S(t,x)=\frac{A(t)}{x_a}+B(t)+C(t)x_a>0, \quad \forall (t,x) \in
[a,b) \times \mathbb{R}^+ \cap r_{x_a}.\] Therefore, by Lemma
\ref{lema1}.i, any solution $x(t,x_a)$ of (\ref{normal}) with
$x(a)=x_a \in J$ is strictly positive for all $t \in I$. As a
consequence, $|x(t,x_a)-x^\ast(t)|>|x^\ast(t)|$ in $I$, which
yields the instability of $x^\ast(t)$.
\end{proof}

\begin{rem} \label{geom} When $t=a$ is a simple zero of $A(t)$, the behavior of the solutions brought out in
Theorem \ref{esta} arises
immediately from the fact that, in such a case, $(a,0)$ is a
saddle point of the corresponding two-dimensional system
(\ref{lat})-(\ref{lax}).
\end{rem}

\section{An example case: the normal form} \label{sec_normalform}

The normal form of Abel equations of the second kind is of the form
\begin{equation} \label{nf}
  x\dot x = A(t)+x.
\end{equation}
This type of ODE is specially important because the generic Abel
equation of the second kind (\ref{normal}) is readily transformable
to (\ref{nf}) with a change of variables that, however, does not
preserve time \cite{p}.

\begin{prop} \label{tnf} Let $A(t)$ be a $\mathcal C^1$, $T$-periodic function. If $A(t)$ has at least
one zero in $[0,T]$ and
\begin{equation} \label{cond22} \min\dot A(t) > -\frac14,
\end{equation}
then $(\ref{nf})$  has, at least, one $T$-periodic solution, $x^\ast(t)$, that has the sign
of $-A(t)$, and it is also $\mathcal C^{1}$ and unstable.
Furthermore, if all the zeros of $A(t)$ are simple, then $x^\ast(t)$ is the unique $T$-periodic solution of (\ref{nf}) with nonconstant sign. Additionally,
if
\begin{equation}\label{nose} \int_0^T A(t)dt=0,\end{equation}
then (\ref{nf}) has no other periodic solutions but $x^\ast(t)$.
\end{prop}

\begin{proof} The first part is straightforward from Theorems \ref{teo},
\ref{p1} and \ref{esta}.

The last part of the statement follows if one can ensure that, when (\ref{nose}) is assumed, (\ref{nf}) has no $T$-periodic solutions with constant sign. For, the change of variables $x \mapsto x^{-1}$ and Theorem 2.1 in \cite{bt} guarantee the nonexistence of positive periodic solutions in (\ref{nf}). An equivalent conclusion for negative periodic solutions follows using $x \mapsto -x^{-1}$.
\end{proof}


\begin{rem} It is worth mentioning that restriction (\ref{cond22}) is sharp.
For, recall from Section \ref{prova} that $t_i$ denotes an element
of $Z$, the set of  time instants where $A(t)$ vanishes. Then,
notice that if $\dot A (t_i) < - 1/4$, the phase plane point
$(t_i,0)$ becomes a focus of (\ref{lat})-(\ref{lax}), which implies
that there can not exist any solution $x(t)$ of (\ref{nf}) such that
$x(t_i)=0$.
\end{rem}





\bibliographystyle{elsarticle-num}



\end{document}